\documentclass[11pt]{amsart}

\usepackage[square,compress,comma, numbers,sort]{natbib}
\usepackage[colorlinks=true, citecolor=red, linkcolor=blue]{hyperref}

\allowdisplaybreaks[4]

\newcommand{\E}[1]{\mathbb{E}\left\{ #1\right\}}

\newcommand{\pk}[1]{\mathbb{P} \left\{ #1 \right \} }

\newcommand{\R}{\mathbb{R}}

\newcommand{\Z}{\mathbb{Z}}

\newcommand{\ep}{\varepsilon}

\DeclareMathOperator*{\cov}{\mathbb{C}ov}
\DeclareMathOperator*{\var}{\mathbb{V}ar}

\newcommand{\convd}{\overset{d}{\to}}

\newcommand{\ind}{\mathbb{I}}

\newcommand{\pad }{\mathcal P_{\alpha}^\delta(a)}

\def\IF{\infty}

\def\bqny#1{\begin{eqnarray*} #1 \end{eqnarray*}}
\def\bqn#1{\begin{eqnarray} #1 \end{eqnarray}}

\newtheorem{theorem}{Theorem}[section]

\newtheorem{lemma}[theorem]{Lemma}

\begin{document}

\title{On the speed of convergence of Piterbarg constants}

\author{Krzysztof Bisewski}
\address{Krzysztof Bisewski, Department of Actuarial Science, University of Lausanne,	UNIL-Dorigny, 1015 Lausanne, Switzerland}
\email{Krzysztof.Bisewski@unil.ch}

\author{Grigori Jasnovidov}
\address{Grigori Jasnovidov, Laboratory of Statistical Methods, St. Petersburg Department of Steklov Mathematical Institute of Russian Academy of Sciences,
Department of Actuarial Science, University of Lausanne,	
UNIL-Dorigny, 1015 Lausanne, Switzerland
}
\email{griga1995@yandex.ru}

\bigskip

\date{\today}
 \maketitle

 {\bf Abstract:} In this paper we derive an upper bound for the difference between
the continuous and discrete Piterbarg constants. Our result allows
us to approximate the classical Piterbarg constants by their discrete
counterparts using Monte Carlo simulations with an explicit error rate.\\


\section{Introduction}
Let $B_\alpha(t), t\in\R$ be a fractional Brownian motion (later on fBM), i.e., a
Gaussian process with zero expectation and covariance function given by
\bqny{
\cov(B_\alpha(t),B_\alpha(s)) = \frac{|t|^{\alpha}+|s|^{\alpha}-|t-s|^{\alpha}}{2},
 \ \ \ \ \ t,s \in\R, \ \alpha\in (0,2).
}
One of the constants that typically appear in the asymptotics of
the ruin probabilities for Gaussian processes are the Piterbarg constants
defined for $d>0$ and $\alpha\in (0,2)$ by
\bqny{
\mathcal P_{\alpha}(d,\mathcal K) = \E{\sup\limits_{t\in \R
\cap \mathcal K}
e^{\sqrt 2 B_\alpha(t)-(1+d)|t|^\alpha}},
}

where $\mathcal K$ can be both $[0,\IF)$ and $(-\IF,\IF)$.
We refer to Theorem 10.1 in \cite{20lectures} for finiteness and
positivity of $\mathcal P_{\alpha}(d,\mathcal K)$; some other
contributions dealing with Piterbarg constants are
\cite{longKrzys,ling2019generalized,
ApproximationSupremumMaxstableProcess,secondproj,Debicki_2019pickands-piterbarg}.
We notice, that the Piterbarg constants appear
in Gaussian queueing theory results, see \cite{MR3091101,
grisha_kdebicki_gaussian_reflected}.
As for well-known Pickands constants (see, e.g., \cite{20lectures}),
the exact value of the classical Piterbarg
constant is known only for special case $\alpha=1$, namely
$\mathcal P_{1}(d,(-\IF,\IF)) = 1+\frac{2}{d}-\frac{1}{2d+1}$
and $\mathcal P_{1}(d,[0,\IF)) = 1+\frac{1}{d}$, see \cite{longKrzys} and
\cite{thirdprojectParisian}, respectively.
In other cases naturally arises the question of an approximation of
the classical Piterbarg constants. \newline

Since it seems difficult to simulate fBM
on a continuous-time scale, one can approximate the
 Piterbarg constants by their discrete
analogous defined for $\delta> 0$ by
\bqny{
\mathcal P_{\alpha}^\delta(d,\mathcal K) =
\E{\sup\limits_{t\in G(\delta) \cap \mathcal K}
e^{\sqrt 2 B_\alpha(t)-(1+d)|t|^\alpha}},
}
where $G(\delta) = \delta\Z$ for $\delta>0$. Set in the following
$\mathcal P_{\alpha}^0(d,\mathcal K) = \mathcal
P_{\alpha}(d,\mathcal K)$ and $G(0) = \R$.

The question of speed of convergence of the discrete Piterbarg constants
to continuous ones is related to the estimation of
$\sup_{t\in [0,1]}B_\alpha(t)-\sup_{t\in [0,1]_\delta}B_\alpha(t)$
as $\delta \to 0$.
We refer to \cite{Borovkov18_H<1/2,Borovkov17_H_general} for
the interesting analysis of the expression above. For Brownian motion
case (later on BM), i.e.,
when $\alpha=1$, we refer to \cite{Dieker_survey_BM_sup} for the survey of
the known results for the current moment.

\section{Main Results}

Here we present the results needed for approximation
of $\mathcal P_{\alpha}(d,\mathcal K)$ by simulation
methods. The theorem below derives an upper bound for the difference
between the continuous and discrete Piterbarg constants.
\begin{theorem}\label{piterbarg_const_theo}
For any $d>0$ with some constant $\mathcal C>0$ that does not depend on
$\delta$ it holds, that
as $\delta\to0$
\bqny{
\mathcal P_{\alpha}(d,\mathcal K) - \mathcal P_{\alpha}^\delta(d,\mathcal K)
 \le \mathcal C\delta^{\alpha/2}(-\ln\delta)^{1/2}.
}
\end{theorem}

Since it is impossible to simulate a fBM on the infinite-time horizon
we deal with the truncated version of the Piterbarg constants. Namely,
 for $d>0$ and $\delta,T\ge 0$ a truncated Piterbarg constant
  is defined by
\bqny{
\mathcal P_{\alpha}^\delta(d,\mathcal K;T) =
\E{\sup\limits_{t\in \mathcal K\cap [-T,T]_\delta} e^{
\sqrt 2 B_\alpha(t)-(1+d)|t|^\alpha}},
}
where for real $a,b$ and $\eta> 0$
$[a,b]_\eta = [a,b]\cap \eta \mathbb Z$. Set in the following
 $[a,b]_0 := [a,b]$ for $a<b$.

The following theorem provide us an upper bound for the
difference between a Piterbarg constant and the corresponding
truncated Piterbarg constant:
\begin{theorem}\label{lemma_piterbarg_truncation}
It holds with some uniform in $\delta\ge 0$ constant $\mathcal C$ and
all sufficiently large $T$, that
\bqny{
\mathcal P_{\alpha}^\delta(d,\mathcal K) - \mathcal P_{\alpha}^\delta(d,\mathcal K;T)
\le e^{-\mathcal CT^\alpha}.
}
\end{theorem}
The results above imply the following approach for approximation of
$\mathcal P_{\alpha}(d,\mathcal K)$ for $d>0$:

\begin{itemize}
\item[1)] For small $\delta>0$ take $S_\delta = (-\ln \delta)^{2/\alpha}$ and approximate $\mathcal P_{\alpha}^\delta(d,\mathcal K;S_\delta)$ by
$\widehat{\mathcal P_{\alpha}^\delta}(d,\mathcal K;S_\delta)$
obtained by Monte-Carlo simulations.
\item[2)] Theorem \ref{piterbarg_const_theo} and Theorem
\ref{lemma_piterbarg_truncation} imply that
$\mathcal P_{\alpha}(d,\mathcal K)
-\mathcal P_{\alpha}^\delta(d,\mathcal K;S_\delta) \le \mathcal C\delta^{\alpha/2}(-\ln\delta)^{1/2}$ uniformly for sufficiently small $\delta>0$,
hence we can approximate $\mathcal P_{\alpha}(d,\mathcal K)$ by
$\widehat{\mathcal P_{\alpha}^\delta}(d,\mathcal K;S_\delta)$
with an error term not exceeding
$\mathcal C\delta^{\alpha/2}(-\ln\delta)^{1/2}$ combined with statistical
error from estimation of
$\widehat{\mathcal P_{\alpha}^\delta}(d,\mathcal K;S_\delta)$.
\end{itemize}

Independence of the increments of Brownian motion allows us to
deduce more precise bounds for
$\mathcal P_{1}(d,\mathcal K)-\mathcal P_{1}^\delta(d,\mathcal K)$
than in Theorem \ref{piterbarg_const_theo}. Namely,
\begin{theorem}\label{pit_bm_theo}
For $d>0$ as $\delta \to 0$ it holds, that
\bqny{
\mathcal P_{1}(d,\mathcal K)-\mathcal P_{1}^\delta(d,\mathcal K)
\sim -\frac{\zeta(1/2)}{\sqrt\pi}
\sqrt\delta \cdot \mathcal P_{1}^\delta(d,\mathcal K),
}
where $\zeta$ is the Euler-Riemann zeta function and
$-\frac{\zeta(1/2)}{\sqrt\pi}>0$.
\end{theorem}

It is interesting, that the constant $-\zeta(1/2)/\sqrt \pi$
appears above, the same constant plays the same role in the difference
between the classical Pickands constants, see \cite{pick_speed_convergence}.
This constant appears in many
problems concerning the difference between supremum of BM on a continuous
and discrete grids, see \cite{Dieker_survey_BM_sup}.

\section{Proofs}
Define for $d> 0$ and $\alpha\in (0,2)$
\bqn{\label{Z(t)}
Z_\alpha(t) = \sqrt 2 B_{\alpha}(t)-(d+1)|t|^{\alpha},
\qquad t \in\R .}
Before giving proofs we present and prove several
 lemmas needed for the proofs of
the main results.

\begin{lemma}\label{lemma_tail_behaviour_pick_const}
 For all $d>0$ there exists $C_1,C_2\in\R$ such that
\bqny{
\pk{\sup\limits_{t\in [0,\infty)}e^{Z_\alpha(t)}>x }
\sim C_1 (\ln x)^{C_2} x^{-1-d},\qquad x\to\infty.}
\end{lemma}

\begin{proof}[Proof of Lemma \ref{lemma_tail_behaviour_pick_const}]
By the self-similarity of fBM, for $x>1$ we have
\bqny{
\pk{\sup\limits_{t\in [0,\infty)}e^{Z_\alpha(t)}>x }
&=& \pk{\exists t \in [0,\infty): \sqrt 2 B_\alpha(t)-(d+1)t^{\alpha}>\ln x}
\\&=&
\pk{\exists t \in [0,\infty): \frac{\sqrt 2 B_\alpha(t)}{
(d+1)t^{\alpha}+1}>\sqrt {\ln x}}
=:
\pk{\exists t \in [0,\infty): V(t)>\sqrt {\ln x}}.
}
The variance of $V(t)$ for $t\ge 0$ achieves its unique maximum at
$t_0 = (\frac{1}{\sqrt{d+1}})^{2/\alpha}$ and
$\var\{V(t_0) \} = \frac{1}{2(d+1)},$
hence applying \cite[Theorem~10.1]{20lectures} we find that
there exist some $C_1>0$ and $C_2\in\R$ such that
$\pk{\sup_{t\in [0,\infty)}e^{Z_\alpha(t)}>x } \sim C_1(\ln x)^{C_2} x^{-1-d}$, as
$x\to\infty$.
\end{proof}

\begin{lemma}\label{new_lemma}
For any $p,T>0$ and $\alpha\in (0,2)$ for
sufficiently small $\delta>0$ with some $\mathcal C>0$
that does not depend on $\delta$ it holds, that
\bqny{
\E{\sup\limits_{t,s\in [0,T],|t-s|\le\delta}
|B_\alpha(t)-B_\alpha(s)|^p}^{1/p} \le \mathcal C \delta^{\alpha/2}\sqrt{\ln(|\delta/T|)}
.}
\end{lemma}
\begin{proof}[Proof of Lemma \ref{new_lemma}]
We have by Theorem 4.2 in \cite{Minicourse} with
some non-negative random variable $\mathcal A$ that does
not depend on $\delta$ and has all finite moments that
\bqny{
\E{\sup\limits_{t,s\in [0,T],|t-s|\le\delta}
|B_\alpha(t)-B_\alpha(s)|^p} &=&
T^{\frac{\alpha p}{2}}
\E{\sup\limits_{t,s\in [0,1],|t-s|\le\frac{\delta}{T}}|B_\alpha(t)-B_\alpha(s)|^p}
\\&\le&
T^{\frac{\alpha p}{2}}
\E{\sup\limits_{t,s\in [0,1],|t-s|\le\frac{\delta}{T}}
\left(\mathcal A |t-s|^{\alpha/2} \sqrt{|\ln(|t-s|/T)|}\right)^p}
\\&=& T^{\frac{\alpha p}{2}} (\frac{\delta}{T})^{\alpha p/2}
\sqrt{|\ln(\delta/T)|}^p \E{\mathcal A^p}
\\&=&  C\delta^{\alpha p/2} \sqrt{|\ln(\delta/T)|}^p,
}
and the claim follows.
\end{proof}

Now we are ready to prove the main results.

\begin{proof}[Proof of Theorem \ref{piterbarg_const_theo}]
We have by Theorem \ref{lemma_piterbarg_truncation}
with $T_\delta = (-\ln\delta)^{2/\alpha}$ that
\bqny{
\mathcal P_{\alpha}(d,\mathcal K) - \mathcal P_{\alpha}^\delta(d,\mathcal K)
&\le&
\mathcal P_{\alpha}(d,\mathcal K;T_\delta)-
\mathcal P_{\alpha}^\delta(d,\mathcal K;T_\delta)+o(\delta^{\alpha/2}).
\\&=& \E{\sup\limits_{t\in \mathcal K\cap [-T_\delta,T_\delta]} e^{
Z_\alpha(t)}
-\sup\limits_{t\in \mathcal K\cap [-T_\delta,T_\delta]_\delta} e^{
Z_\alpha(t)}}+o(\delta^{\alpha/2})
\\&\le& 2\E{
\sup\limits_{t\in [0,T_\delta]} e^{ Z_\alpha(t)}
-\sup\limits_{t\in [0,T_\delta]_\delta} e^{Z_\alpha(t)}
}+o(\delta^{\alpha/2}).}

Note that for any $ y\le x$ it holds, that $e^x-e^y\le (x-y)e^x$.
Implementing this inequality we find that, for $s,t\in[0,T_\delta]$,
\bqny{
\left|e^{Z_\alpha(t)}-e^{Z_\alpha(s)}\right|
&\le& e^{\max(Z_\alpha(t),Z_\alpha(s))}
|Z_\alpha(t)-Z_\alpha(s)|
\\&\le&
e^{\max\limits_{w\in [0,\IF)}Z_\alpha(w)}\left|
\sqrt 2(B_\alpha(t)-B_\alpha(s))-(t^\alpha-s^\alpha)\right|.
}
By the lines above we obtain
\bqny{& \ &
\E{\sup\limits_{t\in [0,T_\delta]} e^{ Z_\alpha(t)}
-\sup\limits_{t\in [0,T_\delta]_\delta} e^{Z_\alpha(t)}}
\\&\le&
\E{\sup\limits_{t,s\in [0,T_\delta], |t-s|\le \delta}
| e^{ Z_\alpha(t)}- e^{Z_\alpha(t)}|}
\\&\le&
\E{\sup\limits_{t,s\in [0,T_\delta], |t-s|\le \delta}
e^{\max\limits_{w\in [0,\IF)}Z_\alpha(w)}\left|
\sqrt 2(B_\alpha(t)-B_\alpha(s))-(t^\alpha-s^\alpha)\right|}
\\&\le&
\sqrt 2\E{\sup\limits_{t,s\in [0,T_\delta], |t-s|\le \delta}
e^{\max\limits_{w\in [0,\IF)}Z_\alpha(w)}\left|
B_\alpha(t)-B_\alpha(s)\right|}
+\E{e^{\max\limits_{w\in [0,\IF)}Z_\alpha(w)}}
\sup\limits_{t,s\in [0,T_\delta], |t-s|\le \delta}
\left|t^\alpha-s^\alpha\right|
\\&\le&
\sqrt 2
\E{e^{q\max\limits_{w\in [0,\IF)}Z_\alpha(w)}}^{1/q}
\E{\sup\limits_{t,s\in [0,T_\delta], |t-s|\le \delta}
\left|B_\alpha(t)-B_\alpha(s)\right|^p}^{1/p}+o(\delta^{\alpha/2}),
}
where $1/q+1/p = 1$ and $q$ is chosen such that
$\E{e^{q\max\limits_{w\in [0,\IF)}Z_\alpha(w)}}<\IF$, this is possible to
do by Theorem \ref{lemma_tail_behaviour_pick_const}.
Next by Lemma \ref{new_lemma} and lines above we have
\bqny{
\mathcal P_{\alpha}(d,\mathcal K) - \mathcal P_{\alpha}^\delta(d,\mathcal K)
\le
\mathcal C \delta^{\alpha/2}\sqrt{\ln(|\delta/T_\delta|)}+o(\delta^{\alpha/2})
}
and the claim follows.
\end{proof}

\begin{proof}[Proof of Theorem \ref{lemma_piterbarg_truncation}]
We have
\bqny{
\pad-\mathcal P_\alpha^\delta(d,\mathcal K;T) =
\E{\sup\limits_{t\in \mathcal K \cap(-\infty,\infty)_\delta}e^{Z_\alpha(t)}
-\sup\limits_{t\in \mathcal K \cap[-T,T]_\delta}e^{Z_\alpha(t)}}
\le \E{\sup\limits_{t\in\R}e^{Z_\alpha(t)}\ind(M_T) },
}
where $M_T = \{\omega\in \Omega :
Z_\alpha(\omega,t)>0 \text{ for some } t\in (-\infty,-T]
\cap[T,\infty)\}$ with $\Omega$ being the general probability space.
We have by Borell-TIS inequality and the symmetry
\bqny{
\pk{M_T} \le 2\pk{\exists
t\ge T: \sqrt 2 B_\alpha(t)-t^{\alpha}>0}
= 2\pk{\exists t\ge T :
\frac{\sqrt 2 B_\alpha(t)}{t^{\alpha/2}}>t^{\alpha/2}}
\le e^{-CT^{\alpha}}
.}
By Lemma \ref{lemma_tail_behaviour_pick_const}
$\sup\limits_{t\in\R}e^{Z_\alpha(t)}$ has the finite moment of
order $1+d/2$ and hence applying H\"{o}lder inequality we have
\bqny{
\E{\sup\limits_{t\in\R}e^{Z_\alpha(t)}\ind(M_T) } \le
C_1 \pk{M_T}^{C_2}
\le
e^{-CT^{\alpha}}
}
and the claim follows.\end{proof}

\begin{proof}[Proof of Theorem \ref{pit_bm_theo}]
Let $Z_\alpha(t) = \sqrt{2}B_\alpha(t) - (1+a)|t|^\alpha$ be
defined as in \eqref{Z(t)}. We will
show the statement only for $\mathcal K = (-\IF,\IF)$;
the proof for $\mathcal K = [0,\IF)$
will be analogous (but slightly simpler). Fix $a>0$ for $J\subseteq \R$, we define
\[M^J = \sup_{t\in J} Z_\alpha(t), \quad M^J_\delta = \sup_{t\in J\cap \delta\Z} Z_\alpha(t), \quad \Delta^J_\delta := \frac{M^J-M^J_\delta}{\delta^{1/2}}.\]
For breviety we write $M := M^\R, M_\delta := M^\R_\delta$, and $\Delta_\delta := \Delta^\R_\delta$. Then,
\begin{equation}
\mathcal P_{\alpha}(d,\mathcal K) -
\mathcal P_{\alpha}^\delta(d,\mathcal K) = \E{e^M - e^{M_\delta}}.
\end{equation}
Using mean value theorem and noticing that $M_\delta \leq M$ a.s., we obtain the following bounds
\begin{equation}\label{ineq_piterbarg_diff}
\E{(M-M_\delta) \cdot
	e^{M_\delta}  } \leq
\mathcal P_{\alpha}(d,\mathcal K) -\mathcal P_{\alpha}^\delta(d,\mathcal K)
\leq \E{(M-M_\delta) \cdot
		e^{M}  }.
\end{equation}
Till the end, we want to show that
\begin{equation}\label{eqtoshowpickandslevy}
\lim_{\delta\to0}\E{\Delta_\delta \cdot e^{M_\delta}} =
\lim_{\delta\to0}\E{\Delta_\delta \cdot e^{M}} =
\lim_{\delta\to0} \E{\Delta_\delta} \cdot
\mathcal P_{\alpha}(d,\mathcal K)
 = -\frac{\zeta(1/2)}{\sqrt{\pi}} \cdot \mathcal P_{\alpha}(d,\mathcal K) ,
\end{equation}
which, in combination with \eqref{ineq_piterbarg_diff}, will conclude the proof. We will show \eqref{eqtoshowpickandslevy} in three steps.

{\bf Step 1.} Show that there exists a random variable $D$ such that $\E{D} = -\zeta\big(\tfrac{\alpha-1}{\alpha}\big)\E{\widehat X_1^+}$ and such that for any continuity set $A$ of $D$ and an arbitrary event $B\in \mathcal F$ of positive probability, where $\mathcal F$ is the borel $\sigma$-algebra, under which $Z$ is adapted to the usual fitration, for $J=\R$, as $\delta\downarrow0$
\begin{equation}\label{eqrenyiJ}
\pk{\Delta_\delta^J \in A, B} \to  \pk{D\in A}\pk{B}.
\end{equation}
The statement \eqref{eqrenyiJ} for $J=\R$ implies that $\Delta_\delta\convd D$ and, intuitively, that the sequence is `asymptotically independent' of the process $\{Z_\alpha(t)\}_{t\in\R}$.

Now, \cite[Thm.~4]{ivanovs2018zooming} implies that \eqref{eqrenyiJ} holds for $J=[0,T)$ as well as $J=(-T,0]$ for arbitrary $T\in(0,\infty)$.

Now, let $\tau$ and $\tau_\delta$ be the (almost surely unique) time epochs of the all-time supremum of $Z$ over $\R$ and over $\delta\mathbb \Z$ respectively, i.e. $Z_\alpha(\tau) = M$ and $Z_\alpha(\tau_\delta) = M_\delta$. For any $T>0$, let $C_T^+$ be the event that $\{\tau\in[0,T), \tau_\delta\in[0,T)\}$ and similarly, let $C_T^-:= \{\tau\in(-T,0], \tau_\delta\in(-T,0]\}$. We have
\begin{align*}
\pk{\Delta_\delta \in A,  B} = \pk{\Delta_\delta \in A,  B, C_T^-} + \pk{\Delta_\delta \in A,  B, C_T^+} + \pk{\Delta_\delta \in A,  B, (C_T^- \cup C_T^+)'}\\
= \pk{\Delta^{[0,T)}_\delta \in A,  B, C_T^+}  + \pk{\Delta_\delta^{(-T,0]} \in A,  B,  C_T^-} + \pk{\Delta_\delta \in A,  B, (C_T^- \cup C_T^+)'},
\end{align*}
which implies that
\begin{align*}
\Big|\pk{\Delta_\delta \in A,  B} - \big(\pk{\Delta^{[0,T)}_\delta \in A,  B, C_T^+} + \pk{\Delta^{(-T,0]}_\delta \in A,  B, C_T^-}\big) \Big| \leq \pk{(C_T^- \cup C_T^+)'}.
\end{align*}
We now pass $\delta\to0$, use \eqref{eqrenyiJ} with $B:=B\cap C_T^+$ and $B:=B\cap C_T^-$, and obtain
\begin{align*}
\Big|\pk{\Delta_\delta \in A,  B} - \pk{D\in A}\big(\pk{B, C_T^+} & +\pk{B, C_T^-}\big)\Big| \leq \pk{(C_T^- \cup C_T^+)'}.
\end{align*}

We pass $T\to\infty$ and notice that $C_T^+ \cap C_T^- = \emptyset$ and $\pk{C_T^+\cup C_T^-} \to 1$, as $T\to\infty$, which shows \eqref{eqrenyiJ} for $J=\R$. Finally, we note that $\E{D} = -\frac{\zeta(1/2)}{\sqrt{\pi}}$, which is established in \cite[Corollary~4]{bisewski2020zooming}.

{\bf Step 2.} Show that the family of random variables $\Delta_\delta \cdot e^{M_\delta}$, indexed by $\delta$, is uniformly integrable for $\delta$ small enough.

According to de la Vall{\'e}e-Poussin Theorem, it is enough to establish that there exists $\ep >0$ such that the sequence $\E{(\Delta_\delta \cdot e^{M_\delta})^{1+\ep}}$ is bounded for all $\delta$ small enough. Due to H{\"o}lder's inequality, with $p,q \in [1, \infty)$ satisfying $1/p+1/q=1$ we obtain
\begin{equation}\label{eqHolderUI}
\E{(\Delta_\delta \cdot e^{M_\delta})^{1+\varepsilon}} \leq \left(\E{\Delta_\delta^{p(1+\varepsilon)}}\right)^{1/p} \cdot \left(\E{e^{q(1+\ep)M_\delta}}\right)^{1/q}.
\end{equation}
We now establish that for any $0<p<\alpha$, the $p$th moment of $\Delta_\delta$ is bounded for all $\delta$ small enough. Since $\Delta_\delta \leq \Delta_\delta^{[0, \infty)} + \Delta_\delta^{(-\infty, 0]}$, then for any $p>0$ we have
\begin{align*}
\E{\Delta_\delta}^p & \leq \max\{1, 2^{p-1}\} \cdot \left( \E{\Delta^{[0,\infty)}_\delta}^p + \E{\Delta^{(-\infty, 0]}_\delta}^p \right),
\end{align*}
where we used the fact that $(x+y)^p \leq \max\{1, 2^{p-1}\} (x^p + y^p)$ for all $x,y>0$. The boundedness of $\E{\Delta_\delta}^p$ readily follows from the boundedness of $\E{\Delta^{[0,\infty)}_\delta}^p$ and $\E{\Delta^{(-\infty, 0]}_\delta}^p$, which are established in \cite[Theorem~2]{bisewski2020zooming}. Moreover, since $M$ is exponentially distributed with mean $(1+a)^{-1}$, then $\E{e^{\alpha M}}< \infty$ for any $\alpha < 1+a$. So, in \eqref{eqHolderUI}, if we take $q>1$ and $\ep>0$ small enough such that $q(1+\ep) < 1+a$, then $\E{e^{q(1+\ep)M_\delta}} \leq \E{e^{q(1+\ep)M}}$ will be bounded, as $\delta\to 0$.

{\bf Step 3.} Since $e^{M_\delta} \to e^M$ a.s., and the limit is $\mathcal F$-measurable, then due to the mixing condition in \eqref{eqrenyiJ}, the pair $(\Delta_\delta, e^{M_\delta})$ converges jointly in distribution to $(D, e^M)$, where $D$ is independent of $e^M$. Now, from the continuous mapping theorem we find that $\Delta_\delta \cdot e^{M_\delta} \convd D\cdot e^M$. Finally, \eqref{eqtoshowpickandslevy} follows from the uniform integrability of the family $\Delta_\delta\cdot e^{M_\delta}$.
\end{proof}

\section*{Acknowledgements}
We would like to thank Prof. Enkelejd Hashorva and Prof.
Krzysztof D\c{e}bicki for fruitful discussions. Krzysztof Bisewski's research was funded by SNSF Grant 200021-196888. 
G. Jasnovidov was supported by the Ministry of Science 
and Higher Education of the Russian
Federation, agreement 075-15-2019-1620 
date 08/11/2019 and 075-15-2022-289 date 06/04/2022.

\bibliography{bibliography_speed_of_convergence}{}
\bibliographystyle{apalike}

\end{document}